\documentclass[12pt]{amsproc}

\usepackage{fullpage}
\usepackage[utf8]{inputenc}
\usepackage[T1]{fontenc}
\usepackage{graphicx,subfigure}
\usepackage{amsmath,amsfonts,amssymb,mathtools}
\usepackage{stmaryrd}
\usepackage{mathrsfs,dsfont}
\usepackage{amsthm}
\usepackage{mathabx}
\usepackage{tabularx}
\usepackage{float}

\usepackage{hyperref}

\usepackage{enumerate}

\usepackage{color}

\newcommand{\E}{\mathbb{E}}
\newcommand{\R}{\mathbb{R}}
\newcommand{\N}{\mathbb{N}}

\newtheorem{theo}{Theorem}[section]

\newtheorem{rem}[theo]{Remark}

\newtheorem{propo}[theo]{Proposition}
\newtheorem{lemma}[theo]{Lemma}

\usepackage{algorithm}
\usepackage{algorithmic}

\usepackage{amssymb}
\usepackage{amscd}

\begin{document}

\title{Asymptotic preserving schemes for SDEs driven by fractional Brownian motion in the averaging regime}

\author{Charles-Edouard Br\'ehier}
\address{Univ Lyon, Université Claude Bernard Lyon 1, CNRS UMR 5208, Institut Camille Jordan, 43 blvd. du 11 novembre 1918, F-69622 Villeurbanne cedex, France}
\email{brehier@math.univ-lyon1.fr}

\date{}


\begin{abstract}
We design numerical schemes for a class of slow-fast systems of stochastic differential equations, where the fast component is an Ornstein-Uhlenbeck process and the slow component is driven by a fractional Brownian motion with Hurst index $H>1/2$. We establish the asymptotic preserving property of the proposed scheme: when the time-scale parameter goes to $0$, a limiting scheme which is consistent with the averaged equation is obtained. With this numerical analysis point of view, we thus illustrate the recently proved averaging result for the considered SDE systems and the main differences with the standard Wiener case.
\end{abstract}

\maketitle

\section{Introduction}\label{sec:intro}

Multiscale and stochastic systems propose theoretical and computational challenges in all fields of science, including for instance fluid dynamics, biology, finance and engineering. Many models are driven by the standard Wiener process, however the fractional Brownian motion~\cite{MandelbrotVanNess} is another popular model, especially in finance and turbulence modelling.

In this article, we consider stochastic differential equations (SDEs) of the following type:
\begin{equation}\label{eq:SDEintro}
\left\lbrace
\begin{aligned}
dX^\epsilon(t)&=g(X^\epsilon(t),m^\epsilon(t))d\beta^H(t),\\
dm^\epsilon(t)&=-\frac{1}{\epsilon}m^\epsilon(t)dt+\frac{\sqrt{2}}{\sqrt{\epsilon}}dB(t),
\end{aligned}
\right.
\end{equation}
with $X^\epsilon(0)=x_0$, $m^\epsilon(0)=m_0^\epsilon$, where $\epsilon$ is the time-scale separation parameter, $\bigl(B(t)\bigr)_{t\ge 0}$ is a standard Wiener process, $g:\R^2\to\R$ is a sufficiently regular mapping (see Section~\ref{sec:system} for more details), and $\bigl(\beta^H(t)\bigr)_{t\ge 0}$ is a fractional Brownian motion with Hurst index $H>1/2$, such that $B$ and $\beta^H$ are independent. Note that $m^\epsilon$ is a standard real-valued Ornstein-Uhlenbeck process. Since $H>1/2$, the evolution equation for $X^\epsilon$, {\it i.e.} the first equation in~\eqref{eq:SDEintro}, is understood in the sense of Young, see for instance~\cite{Nualart,NualartRascanu}.

The objective of this article is to build and study the behavior of numerical schemes in the regime $\epsilon\to 0$. In that regime, Hairer and Li proved recently~\cite{HairerLi} that the following averaging principle holds (see Section~\ref{sec:averaging} below): the slow component $X^\epsilon$ converges in probability to the solution $\overline{X}^{(H)}$ of the averaged equation
\begin{equation}\label{eq:averagedintro}
d\overline{X}^{(H)}(t)=\overline{g}(\overline{X}^{(H)}(t))d\beta^H(t),
\end{equation}
with initial condition $\overline{X}^{(H)}(0)=x_0$, where the averaged coefficient is defined by $\overline{g}(x)=\E_{m\sim\mathcal{N}(0,1)}[\overline{g}(x,m)]$. See also~\cite{GehringerLi,LiSieber} for other recent contributions dealing with the behavior of multiscale systems driven by fractional Brownian motion. It is worth mentioning that the expression of the averaging principle differs from the case of SDEs driven by standard Wiener processes in two directions: first, if $H=1/2$, the averaged equation is given by
\[
d\overline{X}^{(1/2)}(t)=\bigl(\overline{g^2}(\overline{X}^{(1/2)}(t))\bigr)^{\frac12}d\beta^{1/2}(t),
\]
second, if $H=1/2$, the convergence only holds in distribution in general. We refer for instance to~\cite{Hasminskii,Veretennikov} and to~\cite[Chapter~17]{PavliotisStuart} for seminal references on the averaging principle for SDEs driven by standard Wiener processes.

In this article, we revisit the averaging principle from~\cite{HairerLi} with a numerical analysis point of view: precisely we consider the notion of asymptotic preserving schemes, see the recent contribution~\cite{BRR} for stochastic systems and references therein. We retrieve the same differences between the standard and fractional Brownian motion cases at the discrete time level. Let us now describe the scheme studied in this article: let $\Delta t$ denote the time-step size, then the scheme is given by
\begin{equation}\label{eq:schemeintro}
\left\lbrace
\begin{aligned}
X_{n+1}^\epsilon&=X_n^\epsilon+g(X_n^\epsilon,m_{n+1}^\epsilon)\bigl(\beta^H(t_{n+1})-\beta^H(t_n)\bigr)\\
m_{n+1}^\epsilon&=e^{-\frac{\Delta t}{\epsilon}}m_n^\epsilon+\sqrt{1-e^{-\frac{2\Delta t}{\epsilon}}}\gamma_n,
\end{aligned}
\right.
\end{equation}
where $\bigl(\gamma_n\bigr)_{n\ge 0}$ are independent standard Gaussian random variables. Note that the Ornstein-Uhlenbeck process $m^\epsilon$ is discretized exactly in distribution: $m_n^\epsilon$ and $m^\epsilon(n\Delta t)$ are equal in distribution for all $n\ge 0$. The slow component $X^\epsilon$ is discretized using a standard Euler scheme (with an implicit treatment of the fast component). It is well-known that the scheme~\eqref{eq:schemeintro} is consistent with the system~\eqref{eq:SDEintro} when $\Delta t\to 0$, when the parameter $\epsilon$ is fixed, see for instance~\cite{HongHuangWang,HuLiuNualart,MishuraShevchenko} for the study of the Euler schemes and variants for SDEs driven by fractional Brownian motion with Hurst index $H>1/2$.

When the time-scale separation parameter vanishes, {\it i.e.} $\epsilon\to 0$, it is straightforward to check that for all $n\ge 1$, one has $m_n^\epsilon\to \gamma_n$, and $X_n^\epsilon\to X_n^0$, where the limiting scheme satisfies
\begin{equation}\label{eq:limitschemeintro}
X_{n+1}^0=X_n^0+g(X_n^0,\gamma_n)\bigl(\beta^H(t_{n+1})-\beta^H(t_n)\bigr).
\end{equation}

The main result of this article is the consistence of the limiting scheme~\eqref{eq:limitschemeintro} with the averaged equation~\eqref{eq:averagedintro}, see Theorem~\ref{th:main} below for a rigorous statement: when $\Delta t\to 0$, $X_{N}^0$ converges in probability to $\overline{X}(T)$, where $T=N\Delta t$. This result means that the scheme~\eqref{eq:schemeintro} is asymptotic preserving: the following diagram commutes when $\Delta t,\epsilon\to 0$
\begin{equation}
\begin{CD}
X_N^\epsilon     @>{\Delta t\to  0}>> X^\epsilon(T) \\
@VV{\epsilon\to 0}V        @VV{\epsilon\to 0}V\\
X_N^0     @>{\Delta t\to 0}>> \overline{X}(T).
\end{CD}
\end{equation}
In practice, this property means that the time-step size $\Delta t$ can be chosen independently of the time-scale separation parameter $\epsilon$, and that the scheme is both consistent with the model when $\epsilon$ is fixed and able to capture the limiting averaged equation. Asymptotic preserving schemes for SDEs driven by standard Wiener processes have been introduced and studied in the recent preprint~\cite{BRR}. To the best of our knowledge, they have not been studied in the case of SDEs driven by fractional Brownian motion. Our study reveals that the differences seen in the averaging principle at the continuous-time level also appear for the discretization: in the fractional Brownian motion case, the convergence holds in probability (instead of in distribution) and in the limiting equation the average of $g$ (instead of $g^2$) is computed.

Observe that proposing an asymptotic preserving scheme in a stochastic context is not trivial. Like in~\cite{BRR}, if the Ornstein-Uhlenbeck component was discretized using a implicit Euler scheme
\[
m_{n+1}^\epsilon=\frac{1}{1+\frac{\Delta t}{\epsilon}}\bigl(m_n^\epsilon+\sqrt{\frac{\Delta t}{\epsilon}}\gamma_n\bigr),
\]
the associated limiting scheme would be given by
\[
X_{n+1}=X_{n}+g(X_n,0)\bigl(\beta^H(t_{n+1})-\beta^H(t_n)\bigr)
\]
instead of~\eqref{eq:limitschemeintro}, which is not consistent with~\eqref{eq:averagedintro} in general: it is consistent only if $\overline{g}(x)=g(x,0)$ for all $x$. In addition, if the Ornstein-Uhlenbeck process $\bigl(m^\epsilon(t)\bigr)_{t\ge 0}$ is replaced by an arbitrary ergodic process, the averaging principle still holds (with an appropriate modification of the definition of $\overline{g}$), however there is no known construction of an asymptotic preserving scheme in this general case, even in the standard Brownian motion case, to the best of our knowledge.

The main result of this article is proved first for a simplified case, assuming that $g(x,m)=g(m)$ only depends on the variable $m$ -- in that case $\overline{g}$ is a constant. The proof in the simplified case is elementary, and it is provided in order to exhibit the main ideas and the main differences compared with the standard Brownian motion case. The main result is then proved in the general case, using more technical arguments.

This article is organized as follows. Section~\ref{sec:setting} presents notation and preliminary results: in particular assumptions concerning the multiscale SDE system are given in Section~\ref{sec:system}, the averaging principle from~\cite{HairerLi} is presented in Section~\ref{sec:averaging}, numerical schemes are presented in Section~\ref{sec:schemes}, and the main result, Theorem~\ref{th:main}, is stated and discussed in Section~\ref{sec:ap}. The simplified case ($g(x,m)=g(m)$) is studied in Section~\ref{sec:simple}, in particular the comparison with the standard Brownian motion case is performed in Section~\ref{sec:brown}. The general case is studied in Section~\ref{sec:general}. Section~\ref{sec:conclusion} is devoted to concluding remarks and perspectives for future works.

\section{General setting}\label{sec:setting}

To simplify notation, in this article we consider real-valued processes. Up to straightforward modifications the results are generalized to higher dimension.

\subsection{Notation}

Let $H\in(\frac12,1)$, let $\bigl(\beta^H(t)\bigr)_{t\ge 0}$ be a real-valued fractional Brownian motion with Hurst index $H$, and let $\bigl(B(t)\bigr)_{t\ge 0}$ be a standard real-valued Brownian motion. It is assumed that $\beta^H$ and $B$ are independent. Let $\mathcal{F}^H=\sigma\{\beta^H(t);~t\ge 0\}$ be the $\sigma$-field generated by the fractional Brownian motion $\beta^H$. The conditional expectation operator $\E[\cdot|\mathcal{F}^H]$ is denoted by $\E^H[\cdot]$ in the sequel.

The time-scale separation parameter is denoted by $\epsilon$, without loss of generality $\epsilon\in(0,1)$. The slow variable is denoted by $X^\epsilon$, whereas the fast variable is denoted by $m^\epsilon$.

\subsection{A criterion for convergence in probability}

In this article, convergence of random variables is understood as convergence in probability -- except in Section~\ref{sec:brown} where convergence in distribution needs to be considered.

Let us state an elementary criterion to express convergence in probability in terms of convergence of averages. We use the following convention throughout this article: a mapping $\varphi:\R\to\R$ is said to be of class $\mathcal{C}_b^K$ for some $K\in\N$ if it is bounded and $K$ times continuously differentiable, and if its derivatives of order $1,\ldots,K$ are bounded.
\begin{lemma}\label{lemma:cv_proba}
Let $(\Omega,\mathcal{F},\mathbb{P})$ be a probability space and $\mathcal{G}\subset\mathcal{F}$ be a $\sigma$-field. Let $K\in\N$.

Let $X$ be a $\mathcal{G}$-measurable real-valued random variable, and let $\bigl(X_N\bigr)_{N\in\N}$ be a sequence of real-valued random variables.

The following statements are equivalent.
\begin{enumerate}
\item[(i)] $X_N$ converges to $X$ in probability when $N\to\infty$: for all $\eta\in(0,1)$,
\[
\mathbb{P}(|X_N-X|>\eta)\underset{N\to\infty}\to 0.
\]
\item[(ii)] For any function $\varphi:\R\to \R$ of class $\mathcal{C}_b^K$, one has
\begin{equation}\label{eq:lemma_cv_proba}
\E\bigl[\big|\E[\varphi(X_N)|\mathcal{G}]-\varphi(X)\big|\bigr]\underset{N\to\infty}\to 0.
\end{equation}
\end{enumerate}
\end{lemma}

The proof of Lemma~\ref{lemma:cv_proba} is given in the appendix.

The expression of convergence in probability in the form of~\eqref{eq:lemma_cv_proba} is convenient for several reasons. First, it allows us to provide proofs using Taylor expansion arguments. Second, in the simplified case ($g$ does not depend on $x$) it allows us to provide an elementary proof and a comparison with the standard Brownian case. Finally, expression~\eqref{eq:lemma_cv_proba} may be appropriate to exhibit a speed of convergence, however this question goes beyond the scope of this article and is left open for future works.

\subsection{The multiscale stochastic system}\label{sec:system}

In this article, we study multiscale stochastic systems of the type
\begin{equation}\label{eq:SDE}
\left\lbrace
\begin{aligned}
dX^\epsilon(t)&=g(X^\epsilon(t),m^\epsilon(t))d\beta^H(t),\\
dm^\epsilon(t)&=-\frac{1}{\epsilon}m^\epsilon(t)dt+\frac{\sqrt{2}}{\sqrt{\epsilon}}dB(t),
\end{aligned}
\right.
\end{equation}
with initial conditions $X^\epsilon(0)=x_0$ and $m^\epsilon(0)=m_0$, which are assumed to be deterministic and independent of the parameter $\epsilon$, for simplicity. Assume that $g:\R^2\to\R$ is of class $\mathcal{C}^3$, with bounded derivatives of order $1,2,3$. Then the system~\eqref{eq:SDE} admits a unique solution, such that for all $t\ge 0$ one has
\[
\left\lbrace
\begin{aligned}
X^\epsilon(t)&=x_0+\int_{0}^{t}g(X^\epsilon(s),m^\epsilon(s))d\beta^H(s),\\
m^\epsilon(t)&=e^{-\frac{t}{\epsilon}}m_0+\sqrt{2}\int_0^te^{-\frac{t-s}{\epsilon}}dB(s).
\end{aligned}
\right.
\]
Note that $m^\epsilon$ is an Ornstein-Uhlenbeck process, for any value of $\epsilon\in(0,1)$. The stochastic integral in the $X^\epsilon$ component is interpreted as a Young integral.

In Section~\ref{sec:simple}, we study a simplified case, where $g(x,m)=g(m)$ for all $x,m\in\R^2$. The general case is studied in Section~\ref{sec:general}

\subsection{The averaging principle}\label{sec:averaging}

The goal of this section is to state the averaging principle result from~\cite{HairerLi}.

Define the averaged coefficient $\overline{g}:\R\to\R$ as follows
\[
\overline{g}(x)=\int g(x,m)d\nu(m),\quad x\in\R,
\]
where $\nu=\mathcal{N}(0,1)$ is the standard Gaussian distribution. Note that $\nu$ is the unique invariant distribution of the Ornstein-Uhlenbeck process $m^\epsilon$, for all $\epsilon\in(0,1)$, and for every fixed $t\in(0,\infty)$ and $x\in\R$, one has
\[
\E[g(x,m^\epsilon(t))]\underset{\epsilon\to 0}\to \overline{g}(x).
\]
The mapping $\overline{g}$ inherits the regularity of the mapping $g$ with respect to the $x$-variable: $\overline{g}$ is of class $\mathcal{C}^3$, with bounded derivatives of order $1,2,3$.

Let $\bigl(\overline{X}(t)\bigr)_{t\ge 0}$ be the unique solution of the averaged equation
\begin{equation}\label{eq:averaged}
d\overline{X}(t)=\overline{g}(\overline{X}(t))d\beta^H(t),
\end{equation}
with initial condition $\overline{X}(0)=x_0$. One has for all $t\ge 0$
\[
\overline{X}(t)=x_0+\int_{0}^{t}\overline{g}(\overline{X}(s))d\beta^H(s),
\]
where the stochastic integral is interpreted in the sense Young sense.

The averaging principle from~\cite{HairerLi} states that $X^\epsilon$ converges in probability to $\overline{X}$ when $\epsilon\to 0$. In this article, we consider a weaker version: for all $T\ge 0$, $X^\epsilon(T)$ converges in distribution to $\overline{X}(T)$. Owing to Lemma~\ref{lemma:cv_proba}, one has the following formulation:
\begin{equation}\label{eq:cvave}
\underset{\epsilon\to 0}\lim~\E\bigl[\big|\E^{H}[\varphi(X^\epsilon(T))]-\varphi(\overline{X}(T))\big|\bigr]=0
\end{equation}
for every function $\varphi$ of class $\mathcal{C}_b^3$.

\subsection{Numerical schemes}\label{sec:schemes}

Let us introduce the numerical scheme studied in this article. Let $T\in(0,\infty)$ and let $\Delta t$ denote the time-step size. We assume that $T=N\Delta t$ where $N\in\N$ is an integer. Set $t_n=n\Delta t$ and $\gamma_n=\Delta t^{-1/2}\bigl(B(t_{n+1})-B(t_n)\bigr)$ for all $n\ge 0$. For any values of the time-scale separation parameter $\epsilon$ and of the time-step size $\Delta t$, the numerical scheme is defined by the recursion
\begin{equation}\label{eq:scheme}
\left\lbrace
\begin{aligned}
X_{n+1}^\epsilon&=X_n^\epsilon+g(X_n^\epsilon,m_{n+1}^\epsilon)\bigl(\beta^H(t_{n+1})-\beta^H(t_n)\bigr)\\
m_{n+1}^\epsilon&=e^{-\frac{\Delta t}{\epsilon}}m_n^\epsilon+\sqrt{1-e^{-\frac{2\Delta t}{\epsilon}}}\gamma_n,
\end{aligned}
\right.
\end{equation}
with $x_0^\epsilon=x_0$ and $m_0^\epsilon=m_0$. To simplify notation, the convention $\delta\beta_n^H=\beta^H(t_{n+1})-\beta^H(t_n)$ is used below.

When $\epsilon\to 0$, with fixed time-step size $\Delta t$, it is straightforward to prove that $X_n^\epsilon\to X_n^0$ (in probability), for all $n\ge 0$, where
\begin{equation}\label{eq:limitscheme}
X_{n+1}^0=X_n^0+g(X_n^0,\gamma_n)\bigl(\beta^H(t_{n+1})-\beta^H(t_n)\bigr),
\end{equation}
with $X_0^0=x_0$.

Let us finally introduce the auxiliary scheme
\begin{equation}\label{eq:auxscheme}
\overline{X}_{n+1}=\overline{X}_n+\overline{g}(\overline{X}_n)\bigl(\beta^H(t_{n+1})-\beta^H(t_n)\bigr),
\end{equation}
for all $n\ge 0$, with $\overline{X}_0=x_0$. Note that the auxiliary scheme~\eqref{eq:auxscheme} is the standard Euler scheme with time-step size $\Delta t$ applied to the averaged equation~\eqref{eq:averaged}.

\subsection{Asymptotic preserving property}\label{sec:ap}

We are now in position to state the main result of this article.
\begin{theo}\label{th:main}
The scheme~\eqref{eq:scheme} is asymptotic preserving: the following diagram commutes
\begin{equation}\label{eq:diagramAP}
\begin{CD}
X_N^\epsilon     @>{\Delta t\to  0}>> X^\epsilon(T) \\
@VV{\epsilon\to 0}V        @VV{\epsilon\to 0}V\\
X_N^0     @>{\Delta t\to 0}>> \overline{X}(T)
\end{CD}
\end{equation}
where convergence is understood as convergence in probability, and $T=N\Delta t$, with arbitrary fixed $T\in(0,\infty)$.

The asymptotic preserving property can be rewritten as follows: for any real-valued mapping $\varphi$ of class $\mathcal{C}_b^3$, one has
\begin{equation}\label{eq:AP}
\underset{\Delta t\to 0}\lim~\underset{\epsilon\to 0}\lim~\E\bigl[\big|\E^{H}[\varphi(X_N^\epsilon)]-\varphi(\overline{X}(T))\big|\bigr]=\underset{\epsilon\to 0}\lim~\underset{\Delta t\to 0}\lim~\E\bigl[\big|\E^{H}[\varphi(X_N^\epsilon)]-\varphi(\overline{X}(T))\big|\bigr]=0.
\end{equation}
\end{theo}
The reformulation~\eqref{eq:AP} is due to the criterion for convergence in probability given by Lemma~\ref{lemma:cv_proba}. In order to prove Theorem~\ref{th:main}, we only need to prove that the limiting scheme is consistent with the averaged equation, {\it i.e.} that the following result holds.
\begin{propo}\label{propo:main}
Let $T\in(0,\infty)$, and let the time-step size $\Delta t$ satisfy $T=N\Delta t$, with $N\in\N$.

Let $\bigl(X_n^0\bigr)_{n\ge 0}$ be given by the limiting scheme~\eqref{eq:limitscheme}, and let $\bigl(\overline{X}_n\bigr)_{n\ge 0}$ be given by the auxiliary scheme~\eqref{eq:auxscheme}, with $X_0^0=\overline{X}_0=x_0$.

For any real-valued mapping $\varphi$ of class $\mathcal{C}_b^3$, one has
\[
\underset{\Delta t\to 0}\lim~\E\big[\big|\E^H[\varphi(X_N^0)]-\varphi(\overline{X}_N)\big|\bigr]=0.
\]
\end{propo}

\begin{rem}
In the simplified case (Section~\ref{sec:simple}), it is sufficient to assume that the functions $\varphi$ are of class $\mathcal{C}_b^2$.
\end{rem}

Let us provide the proof of Theorem~\ref{th:main}, assuming that Proposition~\ref{propo:main} holds.
\begin{proof}[Proof of Theorem~\ref{th:main}]
On the one hand, for fixed $\Delta t$, $X_N^\epsilon$ converges in probability to $X_N^0$ when $\epsilon\to 0$, by construction of the scheme. Note that the auxiliary scheme~\eqref{eq:auxscheme} is consistent with the averaged equation, see for instance~\cite{HongHuangWang,HuLiuNualart,MishuraShevchenko}: when $\Delta t\to 0$, $\overline{X}_N$ converges in probability to $\overline{X}(T)$. Owing to Proposition~\ref{propo:main} and to Lemma~\ref{lemma:cv_proba}, we deduce that $X_N^0$ converges to $\overline{X}(T)$ in probability when $\Delta t\to 0$.

On the other hand, for fixed $\epsilon$, the scheme~\eqref{eq:scheme} is consistent with~\eqref{eq:SDE} when $\Delta t\to 0$, in particular $X_N^\epsilon$ converges in probability to $X^\epsilon(T)$ , see for instance~\cite{HongHuangWang,HuLiuNualart,MishuraShevchenko}. Owing to the averaging principle, one has~\eqref{eq:cvave}, which means that $X^\epsilon(T)$ converges to $\overline{X}(T)$ in probability when $\epsilon\to 0$.

We thus conclude that the diagram~\eqref{eq:diagramAP} commutes, where convergence is understood as convergence in probability, thus the scheme is asymptotic preserving.
\end{proof}

It only remains to prove Proposition~\ref{propo:main}. The proof is given first in the simplified case ($g$ does not depend on $x$) in Section~\ref{sec:simple}, then in the general case in Section~\ref{sec:general}. The proof in the simplified case is elementary and is given for pedagogical reasons and to illustrate the differences with the standard Brownian motion case. The analysis of the general case requires more technical arguments.

\section{Study of the simplified problem}\label{sec:simple}

In this section, we assume that the mapping $g$ only depends on $m$: one has $g(x,m)=g(m)$ for all $(x,m)\in\R^2$. Then the averaged quantity $\overline{g}$ is a constant:
\[
\overline{g}=\int g(m)d\nu(m).
\]
In addition, note that one has $\overline{X}_n=\overline{X}(t_n)=x_0+\overline{g}\beta^H(t_n)$ for all $n\ge 0$.

\begin{rem}
In the simplified case, it is sufficient to assume that $g$ is globally Lipschitz continuous.
\end{rem}

Below, first we provide the proof of Proposition~\ref{propo:main} in this case, second we provide a comparison with the case where the fractional Brownian motion $\beta^H$ with $H>1/2$ is replaced by a standard Brownian motion $\beta$. We illustrate the two main fundamental differences: in the latter case the convergence is understood as convergence in distribution, and the averaged equation is not given by averaging $g$ -- one needs to average $g^2$.

The standard Brownian motion case is already well-understood (see for instance~\cite[Chapter~17]{PavliotisStuart} for the averaging principle and~\cite{BRR} for the design and analysis of asymptotic preserving schemes), however we provide details for pedagocial reasons -- and the presentation differs from~\cite{BRR}.

\subsection{Proof of Proposition~\ref{propo:main} in the simplified case}

\begin{proof}[Proof of Proposition~\ref{propo:main} in the simplified case]
Let us introduce a family of auxiliary random variables: for all $n\in\{0,\ldots,N\}$, set
\[
X_N^{(n)}=x_0+\sum_{k=0}^{n-1}\overline{g}\delta\beta_k^H+\sum_{k=n}^{N-1}g(\gamma_k)\delta\beta_k^H.
\]
Note that by construction, one has $X_N^0=X_N^{(0)}$ and $\overline{X}_N=X_N^{(N)}$. In addition, for all $n\in\{0,\ldots,N-1\}$, set $S_N^{(n)}=x_0+\sum_{k=0}^{n-1}\overline{g}\delta\beta_k^H+\sum_{k=n+1}^{N-1}g(\gamma_k)\delta\beta_k^H$. Then one has
\begin{align*}
X_N^{(n)}&=S_N^{(n)}+g(\gamma_n)\delta\beta_n^H\\
X_N^{(n+1)}&=S_N^{(n)}+\overline{g}\delta\beta_n^H.
\end{align*}
for all $n\in\{0,\ldots,N-1\}$.

Let $\varphi$ be of class $\mathcal{C}_b^2$ (we take $K=2$ when applying Lemma~\ref{lemma:cv_proba} in the simplified case). Observe that $\overline{X}_N$ is $\mathcal{F}^H$-measurable, thus one has $\E^H[\varphi(\overline{X}_N)]=\varphi(\overline{X}_N)$. Using a telescoping sum argument and a second-order Taylor expansion, one obtains
\begin{align*}
\E^H[\varphi(X_N^0)]-\varphi(\overline{X}_N)&=\E^H[\varphi(X_N^{(0)})]-\E^H[\varphi(X_N^{(N)})]\\
&=\sum_{n=0}^{N-1}\bigl(\E^H[\varphi(X_N^{(n)})]-\E^H[\varphi(X_N^{(n+1)})]\bigr)\\
&=\sum_{n=0}^{N-1}\bigl(\E^H[\varphi(S_N^{(n)}+g(\gamma_n)\delta\beta_n^H)]-\E^H[\varphi(S_N^{(n)}+\overline{g}\delta\beta_n^H)]\bigr)\\
&=\sum_{n=0}^{N-1}\E^H[\varphi'(S_N^{(n)})(g(\gamma_n)-\overline{g})]\delta\beta_n^H+\sum_{n=0}^{N-1}{\rm O}(|\delta\beta_n^H|^2).
\end{align*}
On the one hand, since the random variables $\bigl(\gamma_n\bigr)_{0\le n\le N-1}$ are independent, and are independent of $\beta^H$, one has
\[
\E^H[\varphi'(S_N^{(n)})(g(\gamma_n)-\overline{g})]=\E^H[\varphi'(S_N^{(n)})]\E[g(\gamma_n)-\overline{g}]=0
\]
by definition of $\overline{g}$. Indeed, observe that $S_N^{(n)}$ only depends on $\gamma_k$ with $k\neq n$, and on $\beta^H$.

On the other hand, since the Hurst index satisfies $H>1/2$, one has
\[
\sum_{n=0}^{N-1}\E[|\delta\beta_n^H|^2]={\rm O}\bigl(\Delta t^{2H-1}\bigr)\underset{\Delta t\to 0}\to 0.
\]
Gathering the estimates then gives the required convergence result: for all functions $\varphi$ of class $\mathcal{C}_b^2$, one has
\[
\E\bigl[\big|\E^{H}[\varphi(X_N^0)]-\varphi(\overline{X}_N)\big|\bigr]\underset{\Delta t\to 0}\to 0
\]
which concludes the proof of Proposition~\ref{propo:main} in the simplified case.
\end{proof}

\begin{rem}
The proof above provides a rate of convergence $2H-1$, which is consistent with the rate of convergence of the standard Euler scheme for SDEs driven by fractional Brownian motion.
\end{rem}

\subsection{Comparison with the standard Brownian Motion case}\label{sec:brown}

The objective of this section is to provide a comparison with the situation where the fractional Brownian motion $\beta^H$ is replaced by a standard Brownian motion $\beta$ (independent of $B$). We thus consider the system
\begin{equation}\label{eq:SDE_brown}
\left\lbrace
\begin{aligned}
dX^\epsilon(t)&=g(m^\epsilon(t))d\beta(t),\\
dm^\epsilon(t)&=-\frac{1}{\epsilon}m^\epsilon(t)dt+\frac{\sqrt{2}}{\sqrt{\epsilon}}dB(t).
\end{aligned}
\right.
\end{equation}
The associated numerical scheme is defined by
\begin{equation}\label{eq:scheme_brown}
\left\lbrace
\begin{aligned}
X_{n+1}^\epsilon&=X_n^\epsilon+g(m_{n+1}^\epsilon)\bigl(\beta(t_{n+1})-\beta(t_n)\bigr)\\
m_{n+1}^\epsilon&=e^{-\frac{\Delta t}{\epsilon}}m_n^\epsilon+\sqrt{1-e^{-\frac{2\Delta t}{\epsilon}}}\gamma_n,
\end{aligned}
\right.
\end{equation}
with initial conditions $X^\epsilon(0)=X_0^\epsilon=x_0$ and $m^\epsilon(0)=m_0^\epsilon=m_0$.

On the one hand, in that setting the averaging principle holds as follows: $X^\epsilon$ converges in distribution to $\overline{X}$ defined by
\[
\overline{X}(t)=x_0+\bigl(\overline{g^2}\bigr)^{\frac12}\beta(t)
\]
with $\overline{g^2}=\int g^2d\nu$. Note that in general $\overline{g^2}>\overline{g}^2$.

On the other hand, for fixed $\Delta t>0$, one has the convergence result (in probability) $X_n^\epsilon\to X_n^0$ for all $n\in\{0,\ldots,N\}$, where
\[
X_{n+1}^0=X_n^0+g(\gamma_n)\bigl(\beta(t_{n+1})-\beta(t_n)\bigr).
\]
The scheme~\eqref{eq:scheme_brown} is asymptotic preserving, when convergence is understood in distribution.
\begin{propo}
The limiting scheme is consistent, for convergence in distribution, with the averaged equation. More precisely, let $T\in(0,\infty)$ and let $\varphi$ be of class $\mathcal{C}_b^3$. Then
\[
\underset{\Delta t\to 0}\lim~\underset{\epsilon\to 0}\lim~\E[\varphi(X_N^\epsilon)]=\underset{\epsilon\to 0}\lim~\underset{\Delta t\to 0}\lim~\E[\varphi(X_N^\epsilon)]=\E[\varphi(\overline{X}(T))].
\]
\end{propo}

The result above is an immediate consequence of~\cite[Theorem~3.7]{BRR}. However, for pedagogical reasons, we provide a proof of the consistency of the limiting scheme with the averaged equation, using the same approach as in the proof of Proposition~\ref{propo:main} above in the simplified case. This allows us to give a comparison of the standard and fractional Brownian motion cases.

\begin{proof}
As in the proof of Proposition~\ref{propo:main} (in the simplified case) above, introduce the following family of random variables: for all $n\in\{0,\ldots,N\}$ define
\[
X_N^{(n)}=x_0+\sum_{k=0}^{n-1}\bigl(\overline{g^2}\bigr)^{\frac12}\delta\beta_k+\sum_{k=n}^{N-1}g(\gamma_k)\delta\beta_k.
\]
Note that by construction, one has $X_N^0=X_N^{(0)}$ and $\overline{X}(t_N)=X_N^{(N)}$. In addition, for all $n\in\{0,\ldots,N-1\}$, set $S_N^{(n)}=x_0+\sum_{k=0}^{n-1}\bigl(\overline{g^2}\bigr)^{\frac12}\delta\beta_k+\sum_{k=n+1}^{N-1}g(\gamma_k)\delta\beta_k$. Then one has
\begin{align*}
X_N^{(n)}&=S_N^{(n)}+g(\gamma_n)\delta\beta_n\\
X_N^{(n+1)}&=S_N^{(n)}+\bigl(\overline{g^2}\bigr)^{\frac12}\delta\beta_n.
\end{align*}
Let $\varphi$ be of class $\mathcal{C}_b^3$. Using a telescoping sum argument and a third-order Taylor expansion, one obtains
\begin{align*}
\E^H[\varphi(X_N^0)]-\varphi(\overline{X}(t_N))&=\E^H[\varphi(X_N^{(0)})]-\E^H[\varphi(X_N^{(N)})]\\
&=\sum_{n=0}^{N-1}\bigl(\E^H[\varphi(X_N^{(n)})]-\E^H[\varphi(X_N^{(n+1)})]\bigr)\\
&=\sum_{n=0}^{N-1}\bigl(\E^H[\varphi(S_N^{(n)}+g(\gamma_n)\delta\beta_n)]-\E^H[\varphi(S_N^{(n)}+\bigl(\overline{g^2}\bigr)^{\frac12}\delta\beta_n)]\bigr)\\
&=\sum_{n=0}^{N-1}\E^H[\varphi'(S_N^{(n)})\bigl(g(\gamma_n)-\bigl(\overline{g^2}\bigr)^{\frac12}\bigr)\delta\beta_n]\\
&\quad+\frac12\sum_{n=0}^{N-1}\E^H[\varphi'(S_N^{(n)})(g^2(\gamma_n)-\overline{g^2})\delta\beta_n^2]\\
&\quad+\sum_{n=0}^{N-1}{\rm O}(|\delta\beta_n|^3).
\end{align*}
The first order term vanishes in expectation, since the increments of the standard Brownian motion $\beta$ are independent. Indeed, observe that the random variables $S_N^{(n)}$, $\gamma_n$ and $\delta\beta_n$ are independent: for all $n\in\{0,\ldots,N-1\}$
\[
\E\bigl[\E^H[\varphi'(S_N^{(n)})(g(\gamma_n)-\bigl(\overline{g^2}\bigr)^{\frac12})\delta\beta_n]\bigr]=\E[\varphi'(S_N^{(n)})]\E[g(\gamma_n)-\bigl(\overline{g^2}\bigr)^{\frac12}]\E[\delta\beta_n]=0,
\]
using $\E[\delta\beta_n]=0$.

The second order term vanishes almost surely by the definition of the averaged coefficient $\overline{g^2}$: for all $n\in\{0,\ldots,N-1\}$, using the independence of the increments of the standard Brownian motion $\beta$ one has
\[
\E^H[\varphi'(S_N^{(n)})(g^2(\gamma_n)-\overline{g^2})\delta\beta_n^2]=\E^H[\varphi'(S_N^{(n)})]\E[g^2(\gamma_n)-\overline{g^2}]\Delta t=0.
\]
Finally, the last term satisfies $\sum_{n=0}^{N-1}\E|\delta\beta_n|^3={\rm O}\bigl(\Delta t^{\frac12}\bigr)\underset{\Delta t\to 0}\to 0$.

Gathering the estimates gives the required convergence result
\[
\underset{\Delta t\to 0}\lim~\big|\E[\varphi(X_N^0)]-\E[\varphi(\overline{X}(T))]\big|=0
\]
and concludes the proof.
\end{proof}
In the proof above, we can exhibit the two main differences between the fractional and standard Brownian motion case. In the latter case, the first order terms of the Taylor expansion only vanishes in expectation, hence the need to consider convergence in distribution. In addition, the averaging procedure is only visible in the second order terms of the Taylor expansion, hence a different expression of the averaged coefficient, whereas it is visible in the first order term in the fractional Brownian motion case.

\section{Study of the general case}\label{sec:general}

The analysis of the general case requires more involved techniques, compared with the simplified case. We first state the useful auxiliary results in Section~\ref{sec:gen_aux}, and give the proof of Proposition~\ref{propo:main} in Section~\ref{sec:gen_proof}. The proof of an auxiliary result stated in Section~\ref{sec:gen_aux} is given in Section~\ref{sec:gen_lemma}.

\subsection{Auxiliary results}\label{sec:gen_aux}

Let us state the main auxiliary results which are used in Section~\ref{sec:gen_proof} below to prove Proposition~\ref{propo:main}.

Let $T\in(0,\infty)$ be fixed. For all $\alpha\in(0,1)$, if $f:[0,T]\to\R$ is a real-valued function, set
\[
\|f\|_{\alpha}=\underset{0\le t_1<t_2\le T}\sup~\frac{|f(t_2)-f(t_1)|}{|t_2-t_1|}.
\]
Recall that $f$ is $\alpha$-H\"older continuous if $\|f\|_{\alpha}<\infty$.

First, the trajectories of a fractional Brownian motion $\bigl(\beta^{H}(t)\bigr)_{0\le t\le T}$ with Hurst index $H$ are $\alpha$-H\"older continuous, for all $\alpha\in(0,H)$. More precisely, for all $\alpha\in(0,H)$, there exists an almost surely finite random variable $C_\alpha$, such that
\begin{equation}\label{eq:holder_betaH}
\|\beta^H\|_{\alpha}\le C_\alpha,
\end{equation}
moreover $\E[C_\alpha^m]<\infty$ for all $m\in\N$. The property~\eqref{eq:holder_betaH} is a consequence of the Kolmogorov regularity criterion, using the equality $\E[|\beta^H(t_2)-\beta^H(t_1)|^2]=|t_2-t_1|^{2H}$ and the fact that $\beta^H$ is a Gaussian process.

Next, in order to study discrete-time processes, it is convenient to introduce the following variant of the H\"older semi-norms. Let $\Delta t$ denote the time-step size, with the condition $T=N\Delta t$ for some $N\in\N$. For every $t\in[0,T]$, let $\ell(t)\in[0,T]$ be defined by $\ell(t)=n\Delta t$ for all $t\in[n\Delta t,(n+1)\Delta t)$, $n\in\{0,\ldots,N-1\}$ and $\ell(T)=T$. To simplify notation, we omit the dependence of $\ell$ with respect to $\Delta t$. For all $\alpha\in(0,1)$ and $0\le a\le b\le T$, if $f:[0,T]\to\R$ is a real-valued function, set
\[
\|f\|_{a,b,\alpha,\Delta t}=\underset{a\le t_1<t_2\le b; \ell(t_1)=t_1}\sup~\frac{|f(t_2)-f(t_1)|}{|t_2-t_1|}.
\]
Observe that the only change in the definition is the condition $\ell(t_1)=t_1$. The dependence with respect to the left and right hand points of the interval is also made explicit (we only need $a=0$ and $b=T$ for the standard version $\|\cdot\|_{\alpha}$).

In the sequel, we employ the following result to estimate Young integrals.
\begin{lemma}\label{lemma:Young}
Let $z=\big(z(t)\bigr)_{0\le t\le T}$ be a $\alpha$-H\"older continuous real-valued function with $\alpha\in(0,1)$. Let $M\in\N$, and let $F:\R^M\to\R$ be a continuously differentiable function.

Assume that real-valued mappings $y_1,\ldots,y_M$ satisfy $\|y_m\|_{\alpha',\Delta t}<\infty$ for all $m=1,\ldots,M$, for some $\alpha'\in(0,1)$ such that $\alpha+\alpha'>1$.

Then there exists $C\in(0,\infty)$, which does not depend on $\Delta t$, such that the following holds: for all $s,t\in[0,T]$ such that $\ell(s)=s$, one has
\begin{equation}
\begin{aligned}
\big|\int_s^t F(y(\ell(r))dz(r)\big|&\le C\underset{s\le r\le t}\sup~|F(y(\ell(r))| \|z\|_{\alpha}(t-s)^{\alpha}\\
&+\sum_{m=1}^{M}\underset{s\le r\le t}\sup~|\partial_{y_m}F(y(\ell(r))| \|y_m\|_{s,t,\alpha',\Delta t}\|z\|_{\alpha}(t-s)^{\alpha+\alpha'},
\end{aligned}
\end{equation}
where to simplify notation $F(y(\ell(r))=F(y_1(\ell(r)),\ldots,y_M(\ell(r)))$.
\end{lemma}
Lemma~\ref{lemma:Young} is a variant of~\cite[Lemma A.1]{HuLiuNualart}, where the dependence with respect to $\|y_m\|_{s,t\alpha',\Delta t}$ when $m$ varies is made more explicit. This is instrumental in the proof of Lemma~\ref{lemma:aux} below, where bounds for derivatives are proved succesively. We refer to~\cite{HuLiuNualart} for the proof. Note that the standard inequalities for Young integrals, considering the standard H\"older semi-norm $\|\cdot\|_{\alpha'}$ cannot be applied, since $r\mapsto y(\ell(r))$ is piecewise constant, and thus in general is not (H\"older) continuous.

Let us now introduce a generalized version of the auxiliary scheme~\eqref{eq:auxscheme}: for all $n\in\{0,\ldots,N\}$, $k\in\{n,\ldots,N-1\}$ and $x\in\R$, set
\begin{equation}\label{eq:auxscheme_bis}
\begin{aligned}
\overline{X}_{n,k+1}(x)&=\overline{X}_{n,k}(x)+\overline{g}(\overline{X}_{n,k}(x))\delta\beta_k^H\\
\overline{X}_{n,n}(x)&=x.
\end{aligned}
\end{equation}
The definition above is indeed a generalization of~\eqref{eq:auxscheme}: one has $\overline{X}_n=\overline{X}_{0,n}(x_0)$ for all $n\in\{0,\ldots,N\}$.

Finally, let us introduce the auxiliary functions $u_n$, for $n\in\{0,\ldots,N\}$, as follows. Given a real-valued mapping $\varphi:\R\to \R$ of class $\mathcal{C}_b^3$, define
\begin{equation}\label{eq:u}
u_n(x)=\varphi(\overline{X}_{n,N}(x))
\end{equation}
for all $n\in\{0,\ldots,N\}$ and $x\in\R$. To simplify notation, the dependence of $u_n$ with respect to the time-step size $\Delta t$ is omitted. Note that $u_N=\varphi$. Observe that the $u_n$'s are random functions. They satisfy the following property: for all $n\in\{0,\ldots,N-1\}$ and all $x\in\R$, one has
\begin{equation}\label{eq:prop_u}
u_n(x)=u_{n+1}\bigl(x+\overline{g}(x)\delta\beta_n^H\bigr).
\end{equation} 
Indeed, by construction $\overline{X}_{n+1,N}\bigl(x+\overline{g}(x)\delta\beta_n^H\bigr)=\overline{X}_{n,N}(x)$.

We are now in position to state the main auxiliary result of this section.
\begin{lemma}\label{lemma:aux}
Assume that $\varphi$ is of class $\mathcal{C}_b^3$. There exist an almost surely finite random variable $\Lambda$, and an almost surely positive random variable $\Delta t_0\le T$, such that for all $\Delta t\in(0,\Delta t_0)$ one has
\begin{equation}\label{eq:lemmaaux}
\underset{0\le n\le N}\sup~\underset{x\in\R}\sup~\Bigl(|u_n(x)|+|u_n'(x)|+|u_n''(x)|\Bigr)\le \Lambda.
\end{equation}
\end{lemma}

The proof of Lemma~\ref{lemma:aux} is technical and is postponed to Section~\ref{sec:gen_lemma}. The arguments are similar to those used in~\cite{HuLiuNualart} to prove boundedness of the solutions of Euler or Milstein schemes applied to SDEs driven by fractional Brownian motion. Indeed, the proof consists in first expressing the first and second order derivatives of $u_n$ using solutions $\eta_{n,N}$ and $\zeta_{n,N}$ of variation equations (see equations~\eqref{eq:eta} and~\eqref{eq:zeta} below), second in proving appropriate bounds.

\begin{rem}
In the standard Brownian motion case (Section~\ref{sec:brown}), the role of $u_n$ would be played by a function defined as $\E[\varphi(\overline{X}_{n,N}(x))]$. The Markov property would play a key role in the analysis, and by homogeneity it would be sufficient to look at the properties of the mapping $\E[\varphi(\overline{X}_{0,n}(x))]$.

In the fractional Brownian motion case, the Markov property is not satisfied, and taking expectation is not relevant. Instead of the Markov property, the flow property~\eqref{eq:prop_u} is satisfied. Observe that the auxiliary functions $u_n$ need to be random. In addition, $u_n$ needs to be defined in terms of $\overline{X}_{n,N}$, instead of $\overline{X}_{0,N-n}$, since the process is not time homogeneous.

Even if the increments of the fractional Brownian motion $\bigl(\delta\beta_n^H\bigr)_{0\le n\le N-1}$ are not independent, the definition~\eqref{eq:u} of the random functions $u_n$ above makes sense. In the property~\eqref{eq:prop_u}, this lack of independence results in the following property: $u_{n+1}$ is not independent of $\delta\beta_n^H$.
\end{rem}

\subsection{Proof of Proposition~\ref{propo:main} in the general case}\label{sec:gen_proof}

The objective of this section is to provide the proof of Proposition~\ref{propo:main}. The guideline of the proof is similar to the approach in the simplified case, except for the first step:
\begin{itemize}
\item the error is written in terms of the functions $u_n$ defined by~\eqref{eq:u} with a telescoping sum argument
\item then the property~\eqref{eq:prop_u} and a second order Taylor expansion argument are used
\item the first order term vanishes by definition of the averaged coefficient $\overline{g}$
\item the second order term is handled using the boundedness of $u_n$, uniformly in $\Delta t$, obtained in Lemma~\ref{lemma:aux}, using the condition $H>1/2$ for the Hurst index.
\end{itemize}

\begin{proof}[Proof of Proposition~\ref{propo:main} in the general case]

Let $\varphi$ be of class $\mathcal{C}_b^3$, and let $u_n$ be defined by~\eqref{eq:u}, for $n=0,\ldots,N$. Recall that $\E^{H}[\cdot]$ denotes the conditional expectation with respect to the $\sigma$-field $\mathcal{F}^H$ generated by the fractional Brownian motion $\beta^H$. On the one hand, by definition of $u_0$ one has $\varphi(\overline{X}_N)=u_0(x_0)=\E^H[u_0(x_0)]$, where $x_0=\overline{X}_0=X_0^0$ ($u_0$ is a $\mathcal{F}^H$-measurable random variable). On the other hand, one has $\varphi(X_N^0)=u_N(X_N^0)$. Using a telescoping sum argument, one then obtains
\begin{align*}
\E^H[\varphi(X_N^0)]-\varphi(\overline{X}_N)&=\E^{H}[u_N(X_N^0)]-\E^{H}[u_0(X_0^0)]\\
&=\sum_{n=0}^{N-1}\bigl(\E^{H}[u_{n+1}(X_{n+1}^0)]-\E^{H}[u_n(X_n^0)]\bigr)\\
&=\sum_{n=0}^{N-1}\E^{H}[u_{n+1}\bigl(X_n^0+g(X_n^0,\gamma_n)\delta\beta_n^H\bigr)-u_{n+1}\bigl(X_n^0+\overline{g}(X_n^0)\delta\beta_n^H\bigr)\bigr],
\end{align*}
where in the last line we have used the definition~\eqref{eq:limitscheme} of the limiting scheme, and the property~\eqref{eq:prop_u} of the functions $u_n$.

A Taylor expansion argument then gives, for all $n\in\{0,\ldots,N-1\}$
\begin{align*}
\E^{H}\bigl[u_{n+1}\bigl(X_n^0+g(X_n^0,\gamma_n)\delta\beta_n^H\bigr)&-u_{n+1}\bigl(X_n^0+\overline{g}(X_n^0)\delta\beta_n^H\bigr)\bigr]\\
&=\E^{H}[u_{n+1}'(X_n^0)\bigl(g(X_n^0,\gamma_n)-\overline{g}(X_n^0)\bigr)]\delta\beta_n^H+R_n,
\end{align*}
with $|R_n|\le C\sup|u_{n+1}''(\cdot)|(\delta\beta_n^H)^2$, since $g$ is assumed to be bounded.

Assume that $\Delta t$ satisfies the condition $\Delta t\le \Delta t_0$ where the random variable $\Delta t_0$ is given in Lemma~\ref{lemma:aux}. On the one hand, using the H\"older continuity property~\eqref{eq:holder_betaH} of the fractional Brownian motion, with $\alpha\in(\frac12,H)$ and the bound for the second order derivative of $u_n$ stated in Lemma~\ref{lemma:aux}, there exists an almost surely finite constant ${\bf\Lambda}$, which does not depend on $\Delta t$, such that
\[
|R_n|\le {\bf \Lambda}\Delta t^{2\alpha} 
\]
for all $n\in\{0,\ldots,N-1\}$, if $\Delta t\le \Delta t_0$.

On the other hand, using the definition of $\overline{g}$ and conditioning with respect to $X_n^0$, one obtains the key property
\[
\E^{H}[u_{n+1}'(X_n^0)\bigl(g(X_n^0,\gamma_n)-\overline{g}(X_n^0)\bigr)]=0,
\]
which means that the first-order term vanishes for all $n\in\{0,\ldots,N-1\}$.

Finally, one obtains
\[
\big|\E^H[\varphi(X_N^0)]-\varphi(\overline{X}_N)\big|\le \sum_{n=0}^{N-1}|R_n|\le {\bf\Lambda}T\Delta t^{2\alpha-1}\underset{\Delta t\to 0}\to 0,
\]
almost surely, since $\alpha$ is chosen such that $2\alpha>1$.

Since $\varphi$ is assumed to be bounded, by the dominated convergence theorem the almost sure convergence implies
\[
\E\bigl[\big|\E^H[\varphi(X_N^0)]-\varphi(\overline{X}_N)\big|\bigr]\underset{\Delta t\to 0}\to 0.
\]
This holds for all functions $\varphi$ of class $\mathcal{C}_b^3$.

This concludes the proof of Proposition~\ref{propo:main} in the general case.
\end{proof}

\begin{rem}
The proof above also handles the simplified case treated in Section~\ref{sec:simple}, with a slightly different point of view. In that case $u_n(x)=\varphi(x+\overline{g}\delta\beta_n^H)$, and the proof of Lemma~\ref{lemma:aux} is straightforward in this simplified case.
\end{rem}

\subsection{Proof of Lemma~\ref{lemma:aux}}\label{sec:gen_lemma}

It remains to give the proof of Lemma~\ref{lemma:aux}, which first requires to introduce additional notation. By the definition~\eqref{eq:u} of the functions $u_n$, one has the following expressions for $u_n'(x)$ and $u_n''(x)$:
\begin{equation}\label{eq:deriv_u}
\begin{aligned}
u_n'(x)&=\varphi'(\overline{X}_{n,N}(x))\eta_{n,N}(x)\\
u_n''(x)&=\varphi'(\overline{X}_{n,N}(x))\zeta_{n,N}(x)+\varphi''(x)(\overline{X}_{n,N}(x))(\eta_{n,N}(x))^2,
\end{aligned}
\end{equation}
where for all $n\in\{0,\ldots,N-1\}$, $k\in\{n,\ldots,N-1\}$ and $x\in\R$, one has
\begin{equation}\label{eq:eta}
\eta_{n,k+1}(x)=\eta_{n,k}(x)+\overline{g}'(\overline{X}_{n,N}(x))\eta_{n,k}(x)\delta\beta_k^H
\end{equation}
and
\begin{equation}\label{eq:zeta}
\zeta_{n,k+1}(x)=\zeta_{n,k}(x)+\overline{g}'(\overline{X}_{n,N}(x))\zeta_{n,k}(x)\delta\beta_k^H+\overline{g}''(\overline{X}_{n,N}(x))(\eta_{n,k}(x))^2 \delta\beta_k^H
\end{equation}
with initial conditions $\eta_{n,n}(x)=1$ and $\zeta_{n,n}(x)=0$.

The expressions in~\eqref{eq:deriv_u} are obtained by recursion arguments. In the sequel, to simplify notation, we let $\overline{X}_{n,k}=\overline{X}_{n,k}(x)$, $\eta_{n,k}=\eta_{n,k}(x)$ and $\zeta_{n,k}=\zeta_{n,k}(x)$.

Let us introduce auxiliary continuous-time processes $\tilde{X}_n,\tilde{\eta}_n$ and $\tilde{\zeta}_n$, defined on the interval $[t_n,T]$ (with $t_n=n\Delta t$), for all $n\in\{0,\ldots,N-1\}$: for all $k\in\{n,\ldots,N-1\}$, if $t\in[t_k,t_{k+1})$, then
\begin{align*}
\tilde{X}_n(t)&=\overline{X}_{n,k}+\overline{g}(\overline{X}_{n,k}(\beta^H(t)-\beta_H(t_k))\\
\tilde{\eta}_n(t)&=\eta_{n,k}+\overline{g}'(\overline{X}_{n,k})\eta_{n,k}(\beta^H(t)-\beta^H(t_k))\\
\tilde{\zeta}_n(t)&=\zeta_{n,k}+\overline{g}'(\overline{X}_{n,k})\zeta_{n,k}(\beta^H(t)-\beta^H(t_k))+\overline{g}''(\overline{X}_{n,k})(\eta_{n,k})^2(\beta^H(t)-\beta^H(t_k)).
\end{align*}
One has $\tilde{X}_n(t_k)=\overline{X}_{n,k}$, $\tilde{\eta}_n(t_k)=\eta_{n,k}$ and $\tilde{\zeta}_n(t_k)=\zeta_{n,k}$ for all $n\le k\le N$.

We are now in position to prove Lemma~\ref{lemma:aux}.
\begin{proof}[Proof of Lemma~\ref{lemma:aux}]
The proof is divided into three steps, where bounds for $\tilde{X}_n$, $\tilde{\eta}_n$ and $\tilde{\zeta}_n$ are proved successively.

{\bf Step 1.}
The auxiliary process $\tilde{X}_n$ satisfies the following property: for all $t_n\le s\le t\le T$, such that $s=\ell(s)$, one has
\[
\tilde{X}_n(t)-\tilde{X}_n(s)=\int_{s}^{t}F_0(\tilde{X}_n(\ell(r)))d\beta^H(r),
\]
where $F_0=\overline{g}$.

Using Lemma~\ref{lemma:Young} with $\alpha=\alpha'\in(\frac12,H)$, and using the boundedness of 
$\overline{g}$ and $\overline{g}'$, one obtains the inequality
\[
\|\tilde{X}_n\|_{s,t,\alpha,\Delta t}\le C_0(g,\|\beta^H\|_\alpha)\Bigl(1+(t-s)^\alpha \|\tilde{X}_n\|_{s,t,\alpha,\Delta t}\Bigr).
\]

Let $\tau_0$ be a positive random variable, chosen such that
\[
C_0(g,\|\beta^H\|_\alpha)\tau_0^\alpha\le \frac12,
\]
and $\tau_0\le T$.

Then if $s=\ell(s)\le t$ satisfy $t-s\le \tau_0$, one obtains
\[
\|\tilde{X}_n\|_{s,t,\alpha,\Delta t}\le 2C_0(g,\|\beta^H\|_\alpha),
\]
which gives
\[
|\tilde{X}_n(t)|\le |\tilde{X}_n(s)|+2\tau_0^\alpha C_0(g,\|\beta^H\|_\alpha).
\]
If the time step size $\Delta t$ satisfies $\Delta t\le\tau_0$, one thus obtains
\[
|\overline{X}_{n,k_2}|\le |\overline{X}_{n,k_1}|+2\tau_0^\alpha C_0(g,\|\beta^H\|_\alpha)
\]
for integers $k_1<k_2$ such that $(k_2-k_1)\Delta t\le \tau_0$, and iterating the argument and using the condition $\overline{X}_{n,n}=x$, one obtains
\[
|\overline{X}_{n,k}|\le |x|+2N_0\tau_0^\alpha C_0(g,\|\beta^H\|_\alpha)
\]
where the integer $N_0$ is chosen such that $N_0\tau_0\ge T$, for all $k\in\{n,\ldots,N\}$, if $\Delta t\le \tau_0$.

{\bf Step 2.}
The auxiliary process $\tilde{\eta}_n$ satisfies the following property: for all $t_n\le s\le t\le T$, such that $s=\ell(s)$, one has
\[
\tilde{\eta}_n(t)-\tilde{\eta}_n(s)=\int_{s}^{t}F_1(\tilde{X}_n(\ell(r)),\tilde{\eta}_n(\ell(r)))d\beta^H(r),
\]
where $F_1(x,\eta)=\overline{g}'(x)\eta$. Note that $|F_1(x,\eta)|+|\partial_x F_1(x,\eta)|\le C|\eta|$ and $|\partial_\eta F_1(x,\eta)|\le C$.

Using Lemma~\ref{lemma:Young} with $\alpha=\alpha'\in(\frac12,H)$, and using the boundedness of 
$\overline{g}$, $\overline{g}'$ and $\overline{g}''$, one obtains the inequality
\begin{align*}
\|\tilde{\eta}_n\|_{s,t,\alpha,\Delta t}&\le C_1(g)\|\beta^H\|_\alpha \|\tilde{\eta}_n\|_{s,t,\infty}\\
&+C_1(g)\|\beta^H\|_{\alpha}(t-s)^\alpha\Bigl(\|\tilde{\eta}_n\|_{s,t,\infty}\|\tilde{X}_n\|_{s,t,\alpha,\Delta t}+\|\tilde{\eta}_n\|_{s,t,\alpha,\Delta t}\Bigr),
\end{align*}
where $\|\tilde{\eta}_n\|_{s,t,\infty}=\underset{r\in[s,t]}\sup~|\tilde{\eta}_n(r)|$.

Owing to Step 1, if $t-s\le \tau_0$, one has $\|\tilde{X}_n\|_{s,t,\alpha,\Delta t}\le 2C_0(g,\|\beta^H\|_\alpha)$. In addition, one has $\|\tilde{\eta}_n\|_{s,t,\infty}\le |\tilde{\eta}_n(s)|+(t-s)^\alpha\|\eta\|_{s,t,\alpha,\Delta t}$.

One thus obtains an inequality of the type
\[
\|\tilde{\eta}_n\|_{s,t,\alpha,\Delta t}\le C_1(g,\|\beta\|_{\alpha})\Bigl(|\tilde{\eta}_n(s)|+(t-s)^\alpha\|\tilde{\eta}_n\|_{s,t,\alpha,\Delta t}\Bigr)
\]
if $t\ge s=\ell(s)$ and $t-s\le \tau_0$. Let $\tau_1$ be a positive random variable, chosen such that
\[
C_1(g,\|\beta\|_\alpha)\tau_1^\alpha\le \frac12
\]
and $\tau_1\le \tau_0$.

Then if $s=\ell(s)\le t$ satisfy $t-s\le \tau_1$, one obtains
\[
\|\tilde{\eta}_n\|_{s,t,\alpha,\Delta t}\le 2C_1(g,\|\beta\|_{\alpha})|\tilde{\eta}_n(s)|,
\]
which gives
\[
|\tilde{\eta}_n(t)|\le \bigl(1+2C_1(g,\|\beta\|_{\alpha})\tau_1^\alpha\bigr)|\tilde{\eta}_n(s)|.
\]
If the time step size satisfies $\Delta t\le \tau_1$, one thus obtains
\[
|\eta_{n,k_2}|\le \bigl(1+2C_1(g,\|\beta\|_{\alpha})\tau_1^\alpha\bigr)|\eta_{n,k_1}|
\]
for integers $k_1<k_2$ such that $(k_2-k_1)\Delta t\le \tau_1$, and iterating the argument and using the condition $\eta_{n,n}=1$ one obtains
\[
|\eta_{n,k}|\le \bigl(1+2C_1(g,\|\beta\|_{\alpha})\tau_1^\alpha\bigr)^{N_1}
\]
where the integer $N_1$ is chosen such that $N_1\tau_1\ge T$, for all $k\in\{n,\ldots,N\}$, if $\Delta t\le \tau_1$. This implies the uniform bound
\begin{equation}\label{eq:boundeta}
\|\tilde{\eta}_n\|_{t_n,T,\infty}\le C_1'(g,\|\beta^H\|_\alpha),
\end{equation}
which holds for all $\Delta t\le \tau_1$ and all $n\in\{0,\ldots,N\}$.

{\bf Step 3.}
The auxiliary process $\tilde{\zeta}_n$ satisfies the following property: for all $t_n\le s\le t\le T$, such that $s=\ell(s)$, one has
\[
\tilde{\zeta}_n(t)-\tilde{\zeta}_n(s)=\int_{s}^{t}F_2(\tilde{X}_n(\ell(r)),\tilde{\eta}_n(\ell(r)),\tilde{\eta}_n(\ell(r)))d\beta^H(r),
\]
where $F_2(x,\eta,\zeta)=\overline{g}'(x)\zeta+\overline{g}''(x)\eta^2$. Note that $|F_2(x,\eta,\zeta)|+|\partial_xF_2(x,\eta,\zeta)|\le C|\zeta|+C|\eta|^2$, $|\partial_\eta F_2(x,\eta,\zeta)|\le C|\eta|$ and $|\partial_\zeta F_2(x,\eta,\zeta)|\le C$.

Using Lemma~\ref{lemma:Young} with $\alpha=\alpha'\in(\frac12,H)$, and using the boundedness of 
$\overline{g}$, $\overline{g}'$, $\overline{g}''$ and $\overline{g}^{(3)}$, one obtains the inequality
\begin{align*}
\|&\tilde{\zeta}_n\|_{s,t,\alpha,\Delta t}\le C_1(g)\|\beta^H\|_\alpha (\|\tilde{\zeta}_n\|_{s,t,\infty}+\|\tilde{\eta}_n\|_{s,t,\infty}^2)\\
&+C_2(g)\|\beta^H\|_{\alpha}(t-s)^\alpha\Bigl(\bigl(\|\tilde{\zeta}_n\|_{s,t,\infty}+\|\tilde{\eta}_n\|_{s,t,\infty}^2\bigr)\|\tilde{X}_n\|_{s,t,\alpha,\Delta t}+\|\tilde{\eta}_n\|_{s,t,\infty}\|\tilde{\eta}_n\|_{s,t,\alpha,\Delta t}+\|\tilde{\zeta}_n\|_{s,t,\alpha,\Delta t}\Bigr),
\end{align*}
where $\|\tilde{\eta}_n\|_{s,t,\infty}=\underset{r\in[s,t]}\sup~|\tilde{\eta}_n(r)|$ and $\|\tilde{\zeta}_n\|_{s,t,\infty}=\underset{r\in[s,t]}\sup~|\tilde{\zeta}_n(r)|$.

Owing to Steps 1 and 2, if $t-s\le \tau_1$, one has $\|\tilde{X}_n\|_{s,t,\alpha,\Delta t}\le C_0'(g,\|\beta^H\|_\alpha)$ and $\|\tilde{\eta}_n\|_{s,t,\infty}\le C_1'(g,\|\beta^H\|_\alpha)$, and one has the uniform bound $\|\tilde{\eta}_n\|_{t_n,T,\infty}\le C_1'(g,\beta^H\|_\alpha)$ if $\Delta t\le \tau_1$. In addition, one has $\|\tilde{\zeta}_n\|_{s,t,\infty}\le |\tilde{\zeta}_n(s)|+(t-s)^\alpha\|\zeta\|_{s,t,\alpha,\Delta t}$.

One thus obtains an inequality of the type
\[
\|\tilde{\zeta}_n\|_{s,t,\alpha,\Delta t}\le C_2(g,\|\beta\|_{\alpha})\Bigl(1+|\tilde{\zeta}_n(s)|+(t-s)^\alpha\|\tilde{\zeta}_n\|_{s,t,\alpha,\Delta t}\Bigr)
\]
if $t\ge s=\ell(s)$ and $t-s\le \tau_0$. Let $\tau_2$ be a positive random variable, chosen such that
\[
C_2(g,\|\beta\|_\alpha)\tau_2^\alpha\le \frac12
\]
and $\tau_2\le \tau_1$.

Then if $s=\ell(s)\le t$ satisfy $t-s\le \tau_1$, one obtains
\[
\|\tilde{\zeta}_n\|_{s,t,\alpha,\Delta t}\le 2C_2(g,\|\beta\|_{\alpha})(1+|\tilde{\zeta}_n(s)|),
\]
which gives
\[
|\tilde{\zeta}_n(t)|\le \bigl(1+2C_2(g,\|\beta\|_{\alpha})\tau_2^\alpha\bigr)|\tilde{\zeta}_n(s)|+2C_2(g,\|\beta\|_{\alpha})\tau_2^\alpha.
\]
If the time step size satisfies $\Delta t\le \tau_2$, one thus obtains
\[
|\zeta_{n,k_2}|\le \bigl(1+2C_2(g,\|\beta\|_{\alpha})\tau_2^\alpha\bigr)|\zeta_{n,k_1}|+2C_2(g,\|\beta\|_{\alpha})\tau_2^\alpha
\]
for integers $k_1<k_2$ such that $(k_2-k_1)\Delta t\le \tau_1$, and iterating the argument and using the condition $\zeta_{n,n}=0$ one obtains
\[
|\zeta_{n,k}|\le \bigl(1+2C_2(g,\|\beta\|_{\alpha})\tau_2^\alpha\bigr)^{N_2}
\]
where the integer $N_2$ is chosen such that $N_2\tau_2\ge T$, for all $k\in\{n,\ldots,N\}$, if $\Delta t\le \tau_2$. This implies the uniform bound
\begin{equation}\label{eq:boundzeta}
\|\tilde{\zeta}_n\|_{t_n,T,\infty}\le C_2'(g,\|\beta^H\|_\alpha),
\end{equation}
which holds for all $\Delta t\le \tau_1$ and all $n\in\{0,\ldots,N\}$.

{\bf Conclusion}

Owing to the expressions~\eqref{eq:deriv_u} for $u_n'(x)$ and $u_n''(x)$, and to the bounds~\eqref{eq:boundeta} and~\eqref{eq:boundzeta} for $\eta_{n,N}=\tilde{\eta}_n(T)$ and $\zeta_{n,N}=\tilde{\zeta}_n(T)$, one obtains
\[
\underset{x\in\R}\sup~\bigl(|u_n(x)|+|u_n'(x)|+|u_n''(x)|\bigr)\le \Lambda
\]
for all $n\in\{0,\ldots,N\}$, where $\Lambda$ is an almost surely finite random variable, if $\Delta t\le \Delta t_0=\tau_2$ where $\tau_2$ is the positive random variable constructed in Step 3 above.

This concludes the proof of~\eqref{eq:lemmaaux}.

\end{proof}

\section{Discussion}\label{sec:conclusion}

In this article, we have studied a class of Euler schemes~\eqref{eq:schemeintro} for slow-fast systems of stochastic differential equations~\eqref{eq:SDEintro}. The slow component $X^\epsilon$ is driven by a fractional Brownian motion with Hurst index $H>1/2$, and converges in probability to a process $\overline{X}$, owing to the averaging principle recently proved in~\cite{HairerLi}. We have proved that well-chosen numerical schemes are able to reproduce a version of the averaging principle at the discrete-time level: they satisfy the asymptotic preserving property stated in Theorem~\ref{th:main}. In particular, the time-step size $\Delta t$ can be chosen independently of the time scale separation parameter $\epsilon$. We have illustrated the differences with the standard Brownian motion case treated in the recent work~\cite{BRR}.

We have left open the important question of proving error estimates: is it possible to prove a uniform accuracy property (as in~\cite{BRR} in the standard Brownian case), {\it i.e.} an error estimate depending on $\Delta t$, uniform with respect to $\epsilon$? Studying this question may require more involved techniques (such as the ones employed in~\cite{HuLiuNualart} and references therein) than those used in this manuscript. More precisely, it is challenging to prove error estimates for
\[
\E\bigl[\big|\E^{H}[\varphi(X_N^\epsilon)]-\E ^{H}[\varphi(\overline{X}_N^0)]\big|\bigr]
\]
when $\epsilon\to 0$, for fixed $\Delta t>0$, which are uniform with respect to $\Delta t$.

Note that obtaining nice error estimates when $\Delta t\to 0$ and/or $\epsilon\to 0$ may provide an alternative proof of the averaging principle from~\cite{HairerLi} by a temporal discretization technique (similar to the one proposed in the seminal article~\cite{Hasminskii} in the standard Brownian motion case). In this article, we have assumed that the Hurst index satisfies $H>1/2$. To the best of our knowledge, the validity and the expression of the averaging principle for SDEs driven by a fractional Brownian motion with Hurst index $H<1/2$ is not known. The construction of well-chosen numerical schemes, with associated nice error estimates, may provide a strategy to generalize the averaging principle to the case $H<1/2$. The scheme would then be asymptotic preserving -- but stating this property needs to identify the limit at the theoretical level. We leave the challenging question of the generalization for $H<1/2$ for future works.

As mentioned at the beginning of Section~\ref{sec:setting}, it is straightforward to generalize the results of this manuscript to higher dimensional situation. It would also be straightforward to consider systems of the type
\[
\left\lbrace
\begin{aligned}
dX^\epsilon(t)&=f(X^\epsilon,m^\epsilon)dt+\sum_{j=1}^{m}g_j(X^\epsilon(t),m^\epsilon(t))d\beta_j^H(t),\\
dm^\epsilon(t)&=-\frac{1}{\epsilon}m^\epsilon(t)dt+\frac{\sqrt{2}}{\sqrt{\epsilon}}\sigma(X^\epsilon(t))dB(t),
\end{aligned}
\right.
\]
where $X^\epsilon(t)\in\R^d$, $\sigma_j:\R^d\times\R\to\R^d$, and $\beta_j^H$ are independent fractional Brownian motions. In particular the treatment of the drift term $f$ has been performed in~\cite{BRR}, this is why assuming that $f(x,m)=0$ in the present article is legitimate, in order to focus on the main features due to the fractional Brownian motion. Assuming that $\sigma(x)=1$ also simplifies the presentation, however it is straightforward to check that the proof of Proposition~\ref{propo:main} remains unchanged (up to modifying notation), thus Theorem~\ref{th:main} also holds in this case.

Instead of using the standard Euler scheme~\eqref{eq:schemeintro} to discretize the slow component $X^\epsilon$ of~\eqref{eq:SDEintro}, in order to increase the performance it would be tempting to employ the scheme studied in~\cite{HuLiuNualart}, of the type
\[
X_{n+1}^\epsilon=X_n^\epsilon+g(X_n^\epsilon,m_{n+1}^\epsilon)\bigl(\beta^H(t_{n+1})-\beta^H(t_n)\bigr)+\frac12\bigl(\nabla_x gg\bigr)(X_n^\epsilon,m_{n+1}^\epsilon)\Delta t^{2H}.
\]
The definition of $m_n^\epsilon$ is not modified from~\eqref{eq:schemeintro}. When $\epsilon\to 0$, one obtains the limiting scheme
\[
X_{n+1}^0=X_n^0+g(X_n^0,\gamma_n)\bigl(\beta^H(t_{n+1})-\beta^H(t_n)\bigr)+\frac12\bigl(\nabla_x g g\bigr)(X_n^\epsilon,\gamma_n)\Delta t^{2H}.
\]
This scheme is consistent (when $\Delta t\to 0$) with the solution of the averaged equation, however it is not clear that the good performance of the modified scheme is preserved when $\epsilon=0$: we have $\E[g(\cdot,\gamma_n)]=\overline{g}(\cdot)$, however in general one may have $\E[\bigl(\nabla_x g g\bigr)(\cdot,\gamma_n)]\neq \bigl(\nabla_x \overline{g} \overline{g}\bigr)(\cdot)$ -- however observe that the equality holds if $d=1$. The construction of more efficient methods than the standard Euler scheme when both $\epsilon>0$ and $\epsilon=0$ may be investigated in future works.

\section*{Acknowledgments}

The author would like to thank Martin Hairer for discussions when this work was initiated, and Shmuel Rakotonirina-Ricquebourg for discussions concerning the proof of Lemma~\ref{lemma:cv_proba}.

\begin{appendix}
\section{Proof of Lemma~\ref{lemma:cv_proba}}

\begin{proof}
{\it Proof that (i)$\Rightarrow$(ii).} Assume that $X_N\to X$ in probability. For any function $\varphi$ of class $\mathcal{C}_b^K$, $\varphi$ is bounded and Lipschitz continuous, thus there exists $C(\varphi)\in(0,\infty)$ such that
\begin{align*}
\E\bigl[\big|\E[\varphi(X_N)|\mathcal{G}]-\varphi(X)\big|\bigr]&\le \E\bigl[\big|\varphi(X_N)-\varphi(X)\big|\bigr]\\
&\le C(\varphi)\E\bigl[\min(|X_N-X|,1)\bigr]\underset{N\to\infty}\to 0,
\end{align*}
where the last step is a consequence of convergence in probability. As a consequence (i) implies (ii).

{\it Proof that (ii)$\Rightarrow$(i).} Introduce an auxiliary function $\varphi:\R\to\R$ be such that
\begin{itemize}
\item $\varphi$ is of class $\mathcal{C}^\infty$
\item $\varphi(x)=0$ if $|x|\le \frac12$
\item $\varphi(x)=0$ if $|x|\ge 1$
\item $0\le \varphi(x)\le 1$ if $\frac12\le |x|\le 1$.
\end{itemize} 

For every $k\in\mathbb{Z}$ and $\eta\in(0,1)$, introduce the interval $I_{k,\eta}=[\frac{k\eta}{2},\frac{(k+1)\eta}{2})$, $m_{k,\eta}=\frac{2k\eta+1}{4}$ and the function
\[
\varphi_{k,\eta}=\varphi\bigl(\eta^{-1}(\cdot-m_{k,\eta})\bigr).
\]
Then $\varphi_{k,\eta}$ is of class $\mathcal{C}_b^K$ (where $K$ is an arbitrary integer).

For all $\eta\in(0,1)$, one has
\[
\mathbb{P}(|X_N-X|>\eta)=\sum_{k\in\mathbb{Z}}\E\bigl[\mathds{1}_{|X_N-X|>\eta}\mathds{1}_{X\in I_{k,\eta}}\bigr].
\]
Owing to the dominated convergence theorem -- one has $\E\bigl[\mathds{1}_{|X_N-X|>\eta}\mathds{1}_{X\in I_{k,\eta}}\bigr]\le\mathbb{P}(X\in I_{k,\eta})$ with $\sum_{k\in\mathbb{Z}}\mathbb{P}(X\in I_{k,\eta})=1$ -- it suffices to prove that for all $k\in\mathbb{Z}$ one has
\[
\E\bigl[\mathds{1}_{|X_N-X|>\eta}\mathds{1}_{X\in I_{k,\eta}}\bigr]\underset{N\to\infty}\to 0.
\]
Note that combining the conditions $|X_N-X|>\eta$ and $X\in I_{k,\eta}$ implies that $\varphi_{k,\eta}(X_N)=1$ and $\varphi_{k,\eta}(X)=0$. Using the fact that $X$ is $\mathcal{G}$-measurable, one thus obtains
\begin{align*}
\E\bigl[\mathds{1}_{|X_N-X|>\eta}\mathds{1}_{X\in I_{k,\eta}}\bigr]&\le \E\bigl[\varphi_{k,\eta}(X_N)\mathds{1}_{X\in I_{k,\eta}}\bigr]\\
&\le \E\bigl[\E[\varphi_{k,\eta}(X_N)|\mathcal{G}]\mathds{1}_{X\in I_{k,\eta}}\bigr]\\
&\le \E\bigl[\bigl(\E[\varphi_{k,\eta}(X_N)|\mathcal{G}]-\varphi_{k,\eta}(X)\bigr)\mathds{1}_{X\in I_{k,\eta}}\bigr]\\
&\le \E\bigl[\big|\E[\varphi_{k,\eta}(X_N)|\mathcal{G}]-\varphi_{k,\eta}(X)\bigr|\bigr]\underset{n\to\infty}\to 0,
\end{align*}
using the assumption that (ii) holds, with $\varphi=\varphi_{k,\eta}$.

Applying the dominated convergence theorem then gives
\[
\mathbb{P}(|X_N-X|>\eta)\underset{N\to\infty}\to 0,
\]
for all $\eta\in(0,1)$. As a consequence $X_N$ converges to $X$ in probability when $N\to\infty$, and (ii) implies (i).

This concludes the proof.
\end{proof}

\end{appendix}


\end{document}